\documentclass[reqno]{amsart}

\usepackage[latin1]{inputenc}
\usepackage{amssymb}
\usepackage{graphicx}
\usepackage{amscd}
\usepackage[hidelinks]{hyperref}
\usepackage{color}
\usepackage{float}
\usepackage{graphics,amsmath,amssymb}
\usepackage{amsthm}
\usepackage{amsfonts}
\usepackage{latexsym}
\usepackage{epsf}
\usepackage{enumerate}
\usepackage{xifthen}
\usepackage{mathrsfs}
\usepackage{dsfont}
\usepackage{makecell}
\usepackage[FIGTOPCAP]{subfigure}
\usepackage{amsmath}
\allowdisplaybreaks[4]
\usepackage{listings}
\usepackage{etoolbox}
\usepackage{fancyhdr}

 \newtheoremstyle{mytheorem}
 {3pt}
 {3pt}
 {\slshape}
 {}
 {\bfseries}
 {.}
 { }
 {}

\numberwithin{equation}{section}

\theoremstyle{theorem}
\newtheorem{theorem}{Theorem}[section]
\newtheorem{corollary}[theorem]{Corollary}
\newtheorem{lemma}[theorem]{Lemma}

\theoremstyle{definition}

\newtheorem{question}{Question}[section]

\newtheorem{remark}{Remark}[section]

\newcommand{\Keywords}[1]{\ifthenelse{\isempty{#1}}{}{\smallskip \smallskip \noindent \textbf{Keywords}. #1}}
\newcommand{\MSC}[2][2010]{\ifthenelse{\isempty{#2}}{}{\smallskip \smallskip \noindent \textbf{#1MSC}. #2}}
\newcommand{\abstractnote}[1]{\ifthenelse{\isempty{#1}}{}{\smallskip \smallskip \noindent \textsuperscript{\dag}#1}}

\makeatletter
\def\specialsection{\@startsection{section}{1}%
  \z@{\linespacing\@plus\linespacing}{.5\linespacing}%
  {\normalfont}}
\def\section{\@startsection{section}{1}%
  \z@{.7\linespacing\@plus\linespacing}{.5\linespacing}%
  {\normalfont\scshape}}
\patchcmd{\@settitle}{\uppercasenonmath\@title}{\Large\boldmath}{}{}
\patchcmd{\@settitle}{\begin{center}}{\begin{flushleft}}{}{}
\patchcmd{\@settitle}{\end{center}}{\end{flushleft}}{}{}
\patchcmd{\@setauthors}{\MakeUppercase}{\normalsize}{}{}
\patchcmd{\@setauthors}{\centering}{\raggedright}{}{}
\patchcmd{\section}{\scshape}{\large\bfseries\boldmath}{}{}
\patchcmd{\subsection}{\bfseries}{\bfseries\boldmath}{}{}
\renewcommand{\@secnumfont}{\bfseries}
\patchcmd{\@startsection}{\@afterindenttrue}{\@afterindentfalse}{}{}
\patchcmd{\abstract}{\leftmargin3pc}{\leftmargin1pc}{}{}

\def\maketitle{\par
  \@topnum\z@ 
  \@setcopyright
  \thispagestyle{empty}
  \ifx\@empty\shortauthors \let\shortauthors\shorttitle
  \else \andify\shortauthors
  \fi
  \@maketitle@hook
  \begingroup
  \@maketitle
  \toks@\@xp{\shortauthors}\@temptokena\@xp{\shorttitle}%
  \toks4{\def\\{ \ignorespaces}}
  \edef\@tempa{%
    \@nx\markboth{\the\toks4
      \@nx\MakeUppercase{\the\toks@}}{\the\@temptokena}}%
  \@tempa
  \endgroup
  \c@footnote\z@
  \@cleartopmattertags
}
\makeatother

\title{Congruences for partition functions related to mock theta functions}

\author[S. Chern]{Shane Chern}
\address[S. Chern]{Department of Mathematics, The Pennsylvania State University, University Park, PA 16802, USA}
\email{shanechern@psu.edu; chenxiaohang92@gmail.com}

\author[L.-J. Hao]{Li-Jun Hao}
\address[L.-J. Hao]{Center for Combinatorics, LPMC, Nankai University, Tianjin 300071, P.R. China}
\email{haolijun152@163.com}

\date{}

\begin{document}

%

\maketitle

\begin{abstract}

Partitions associated with mock theta functions have received a great deal of attention in the literature. Recently, Choi and Kim derived several partition identities from the third and sixth order mock theta functions. In addition, three Ramanujan-type congruences were established by them. In this paper, we present some new congruences for these partition functions.


\MSC{11P83, 05A17.}
\end{abstract}

\section{Introduction}

A partition of a positive integer $n$ is a finite nonincreasing sequence of positive integers whose sum equals $n$. Furthermore, a partition is called a $t$-core partition if there are no hook numbers being multiples of $t$. Let $a_t(n)$ be the number of $t$-core partitions of $n$. It is known \cite{Garvan-Kim-Stanton-1990} that
\begin{align*}
\sum_{n=0}^\infty a_t(n)q^n=\frac{(q^t;q^t)_\infty^t}{(q;q)_\infty}.
\end{align*}
Here and in what follows, we make use of the standard
$q$-series notation (cf. \cite{Gasper-Rahman-2004}).
\begin{align*}
(a)_n=(a;q)_n&:=\prod_{k=0}^{n-1}(1-aq^{k}),\\
(a)_\infty=(a;q)_{\infty}&:=\prod_{k= 0}^\infty (1-aq^{k}),\\
(a_1,a_2,\cdots,a_m;q)_\infty&:=(a_1;q)_\infty(a_2;q)_\infty\cdots(a_m;q)_\infty.
\end{align*}

In addition, the cubic partition, which was introduced by Chan \cite{Chan-2010-1,Chan-2010-2} and named by Kim \cite{Kim-2011} in connection with Ramanujan's cubic continued fractions, is a 2-color partition where the second color appears only in multiples of $2$. Let $a(n)$ denote the number of cubic partitions of $n$, then its generating function is
\begin{align*}
\sum_{n=0}^\infty a(n)q^n=\frac{1}{(q;q)_\infty(q^2;q^2)_\infty}.
\end{align*}

In his last letter to Hardy \cite[pp.~220--223]{BR1995}, Ramanujan defined 17 functions, which he called mock theta functions. Since then, there has been an intense study of partition interpretations for mock theta functions; see \cite{Andrews-2005, Andrews-Dixit-Schultz-Yee-2015, Andrews-Dixit-Yee-2015, Andrews-Garvan-1989, Andrews-Passary-Sellers-Yee-2017}.

Recently, Choi and Kim \cite{Choi-Kim-2012} obtained the following identity related to the third order mock theta function,
\begin{align*}
\upsilon(q)+\upsilon_3(q,q;q)=2\frac{(q^4;q^4)_\infty^3}{(q^2;q^2)_\infty^2},
\end{align*}
where $\upsilon(q)$ is the third mock theta function and $\upsilon_3(q,q;q)$ is defined
by Choi \cite{Choi-2011},
\begin{align*}
\upsilon(q)=\sum_{n=0}^\infty\frac{q^{n(n+1)}}{(-q;q^2)_{n+1}},\quad
\upsilon_3(q,q;q)=\sum_{n=0}^\infty q^n (-q;q^2)_n.
\end{align*}
They also gave the following identities related to the sixth order mock theta functions,
\begin{align*}
\Psi(q)+2\Psi_-(q)&=3\frac{q(q^6;q^6)_\infty^3}{(q;q)_\infty(q^2;q^2)_\infty},\\
2\rho(q)+\lambda(q)&=3\frac{(q^3;q^3)_\infty^3}{(q;q)_\infty(q^2;q^2)_\infty},
\end{align*}
where $\Psi(q)$, $\Psi_-(q)$, $\rho(q)$ and $\lambda(q)$ are the sixth order mock theta functions,
\begin{align*}
&\Psi(q)=\sum_{n=0}^\infty\frac{(-1)^{n}q^{(n+1)^2}(q;q^2)_{n}}{(-q;q)_{2n+1}},&
\Psi_-(q)=\sum_{n=1}^\infty\frac{q^n(-q;q)_{2n-2}}{(q;q^2)_{n}},\\
&\rho(q)=\sum_{n=0}^\infty\frac{q^{\binom{n+1}{2}}(-q;q)_n}{(q;q^2)_{n+1}},&
\lambda(q)=\sum_{n=0}^\infty\frac{(-1)^nq^n(q;q^2)_n}{(-q;q)_{n}}.
\end{align*}
Meanwhile, they studied three analogous partition functions defined by
\begin{align}
\sum_{n=0}^\infty b(n)q^n&=\frac{(q^4;q^4)_\infty^3}{(q^2;q^2)_\infty^2},\label{bn}\\
\sum_{n=0}^\infty c(n)q^n&=\frac{q(q^6;q^6)_\infty^3}{(q;q)_\infty(q^2;q^2)_\infty},\label{cn}\\
\sum_{n=0}^\infty d(n)q^n&=\frac{(q^3;q^3)_\infty^3}{(q;q)_\infty(q^2;q^2)_\infty},\label{dn}
\end{align}
where $b(n)$ denotes the number of partition pairs $(\lambda,\sigma)$ where $\sigma$ is a partition into distinct even parts and $\lambda$ is a partition into even parts of which 2-modular diagram is 2-core, and both $c(n)$ and $d(n)$ can be regarded as $3$-core cubic partitions.

In this paper, we mainly study Ramanujan-type congruences for these partition functions. This paper is organized as follows. In Sect.~\ref{sec:2}, we introduce some preliminary results. In the next two sections, we will prove some Ramanujan-type congruences for $b(n)$ and $c(n)$, respectively. In Sect.~\ref{sec:dn}, by employing $p$-dissection formulas of
Ramanujan's theta functions $\psi(q)$ and $f(-q)$ established by Cui and Gu \cite{Cui-Gu-2013} as well as $(p,k)$-parameter representations due to Alaca and Williams \cite{AW2010},
we show some congruences for $d(n)$. Finally, we end this paper with several open problems.

\section{Preliminaries}\label{sec:2}

Let $f(a,b)$ be Ramanujan's general theta function given by
\begin{equation*}
f(a,b)=\sum_{n=-\infty}^{\infty}a^{\frac{n(n+1)}{2}}b^{\frac{n(n-1)}{2}},\quad|ab|<1.
\end{equation*}

We now introduce the following Ramanujan's classical theta functions,
\begin{align}
\varphi(q)&:=f(q,q)=\sum_{n=-\infty}^{\infty}q^{n^2}=\frac{f^5_2}{f^2_1f^2_4},\label{fqq}\\
\psi(q)&:=f(q,q^3)=\sum_{n=0}^{\infty}q^{\frac{n(n+1)}{2}}=\frac{f^2_2}{f_1},\label{fqq3}\\
f(-q)&:=f(-q,-q^2)=\sum_{n=-\infty}^{\infty}(-1)^nq^{\frac{n(3n+1)}{2}}=f_1.
\end{align}
One readily verifies
\begin{equation}
\varphi(-q)=\frac{f_1^2}{f_2}.
\end{equation}
Here and in the sequel, we write $f_k:=(q^k;q^k)_\infty$ for positive integers $k$ for convenience.

We first require the following 2-dissections.

\begin{lemma}
It holds that
\begin{align}
\frac{1}{f_1^2}&=\frac{f_8^5}{f_2^5f_{16}^2}+2q\frac{f_4^2f_{16}^2}{f_2^5f_8},\label{eq:f1-2}\\
\frac{f_3}{f_1^3}&=\frac{f_4^6f_6^3}{f_2^9f_{12}^2}+3q\frac{f_4^2f_6f_{12}^2}{f_2^7},\label{eq:f3f13}\\
\frac{f_3^3}{f_1}&=\frac{f_4^3f_6^2}{f_2^2f_{12}}+q\frac{f_{12}^3}{f_4}.\label{eq:f33f1}
\end{align}
\end{lemma}

\begin{proof}
Here \eqref{eq:f1-2} comes from the $2$-dissection of $\varphi(q)$ (cf. \cite[p. 40, Entry 25]{Ber1991}). For \eqref{eq:f3f13} and \eqref{eq:f33f1}, see \cite{Xia-Yao-2013}.
\end{proof}

The following 3-dissections are also necessary.

\begin{lemma}\label{le:3-dis}
It holds that
\begin{align}
\frac{1}{\varphi(-q)}&=\frac{\varphi^3(-q^9)}{\varphi^4(-q^3)}\left(1+2qw(q^3)+4q^2 w^2(q^3)\right),\label{eq:1/-phi--3dis}\\
\frac{1}{\psi(q)}&=\frac{\psi^3(q^9)}{\psi^4(q^3)}\left(\frac{1}{w^2(q^3)}-\frac{q}{w(q^3)}+q^2\right)\label{eq:1/psi--3dis},
\end{align}
where
\begin{equation}
w(q)=\frac{f_1f_6^3}{f_2f_3^3}.
\end{equation}
Furthermore,
\begin{equation}\label{eq:1/f13--3dis}
\frac{1}{f_1^3}=\frac{f_9^3}{f_3^{12}}\left(P^2(q^3)+3qP(q^3)f_9^3+9q^2f_9^6\right),
\end{equation}
where
\begin{equation}\label{eq:P}
P(q)=f_1\left(\frac{\varphi^3(-q^3)}{\varphi(-q)}+4q\frac{\psi^3(q^3)}{\psi(q)}\right).
\end{equation}
\end{lemma}

\begin{proof}
For \eqref{eq:1/-phi--3dis} and \eqref{eq:1/psi--3dis}, see Baruah and Ojah \cite{Baruah-Ojah-2011}. For \eqref{eq:1/f13--3dis}, see Wang \cite{Wang-2016}. Note that Wang \cite{Wang-2016} showed
$$P(q)=f_1\left(1+6\sum_{n\ge 0} \left(\frac{q^{3n+1}}{1-q^{3n+1}}-\frac{q^{3n+2}}{1-q^{3n+2}}\right)\right).$$
We know from \cite[Eqs.~(3.2) and (3.5)]{She1994} that
$$4q\frac{\psi^3(q^3)}{\psi(q)}=4\sum_{n\ge 0}\left(\frac{q^{3n+1}}{1-q^{6n+2}}-\frac{q^{3n+2}}{1-q^{6n+4}}\right),$$
$$\frac{\varphi^3(-q^3)}{\varphi(-q)}=1+2\sum_{n\ge 0}\left(\frac{q^{6n+1}}{1-q^{6n+1}}+\frac{q^{6n+2}}{1-q^{6n+2}}-\frac{q^{6n+4}}{1-q^{6n+4}}-\frac{q^{6n+5}}{1-q^{6n+5}}\right).
$$
Hence \eqref{eq:P} follows immediately by the following trivial identity
$$\frac{x}{1-x^2}=\frac{x}{1-x}-\frac{x^2}{1-x^2}.$$
\end{proof}

Furthermore, we need

\begin{lemma}[{\cite[Theorem 2.1]{Cui-Gu-2013}}]\label{psi-p}
For any odd prime $p$,
$$\psi(q)=q^{\frac{p^{2}-1}{8}}\psi(q^{p^{2}})+\sum_{k=0}^{\frac{p-3}{2}}q^{\frac{k^{2}+k}{2}}
f\left(q^{\frac{p^{2}+(2k+1)p}{2}},q^{\frac{p^{2}-(2k+1)p}{2}}\right).
$$
Furthermore, we claim that for $0\leq k\leq(p-3)/2$,
$$\frac{k^{2}+k}{2} \not \equiv \frac{p^{2}-1}{8} \pmod p.$$
\end{lemma}

\begin{lemma}[{\cite[Theorem 2.2]{Cui-Gu-2013}}]\label{cg13}
For any prime $p\geq5$,
\begin{align*}
f(-q)&=(-1)^{\frac{\pm p-1}{6}}q^{\frac{p^{2}-1}{24}}f(-q^{p^{2}})\\
&\quad\quad+\sum_{\substack{k=-\frac{p-1}{2}\\
k\neq\frac{\pm p-1}{6}}}^{\frac{p-1}{2}}(-1)^{k}q^{\frac{3k^{2}+k}{2}}f\left(-q^{\frac{3p^{2}+(6k+1)p}{2}},-q^{\frac{3p^{2}-(6k+1)p}{2}}\right).
\end{align*}
Furthermore, we claim that for $-(p-1)/2\leq k\leq(p-1)/2$ and
$k\neq(\pm p-1)/6$,
\begin{align*}
\frac{3k^{2}+k}{2} \not\equiv
\frac{p^{2}-1}{24} \pmod{p}.
\end{align*}
Here for any prime $p\geq 5$,
\begin{equation*}
\frac{\pm p-1}{6}:=\left\{\begin{array}{ll}\frac{p-1}{6},&\ p \equiv 1
\pmod{6},\\[5pt]
\frac{-p-1}{6},&\ p \equiv -1 \pmod{6}.\end{array}\right.
\end{equation*}
\end{lemma}

At last, we require the following relations due to Alaca and Williams \cite{AW2010}.

\begin{lemma}\label{le:AW}
Let
$$p=p(q):=\frac{\varphi^2(q)-\varphi^2(q^3)}{2\varphi^2(q^3)},$$
and
$$k=k(q):=\frac{\varphi^3(q^3)}{\varphi(q)}.$$
Then
\begin{align*}
f_1&=2^{-\frac{1}{6}}q^{-\frac{1}{24}}p^{\frac{1}{24}}(1-p)^{\frac{1}{2}}(1+p)^{\frac{1}{6}}(1+2p)^{\frac{1}{8}}(2+p)^{\frac{1}{8}}k^{\frac{1}{2}},\\
f_2&=2^{-\frac{1}{3}}q^{-\frac{1}{12}}p^{\frac{1}{12}}(1-p)^{\frac{1}{4}}(1+p)^{\frac{1}{12}}(1+2p)^{\frac{1}{4}}(2+p)^{\frac{1}{4}}k^{\frac{1}{2}},\\
f_3&=2^{-\frac{1}{6}}q^{-\frac{1}{8}}p^{\frac{1}{8}}(1-p)^{\frac{1}{6}}(1+p)^{\frac{1}{2}}(1+2p)^{\frac{1}{24}}(2+p)^{\frac{1}{24}}k^{\frac{1}{2}},\\
f_4&=2^{-\frac{2}{3}}q^{-\frac{1}{6}}p^{\frac{1}{6}}(1-p)^{\frac{1}{8}}(1+p)^{\frac{1}{24}}(1+2p)^{\frac{1}{8}}(2+p)^{\frac{1}{2}}k^{\frac{1}{2}},\\
f_6&=2^{-\frac{1}{3}}q^{-\frac{1}{4}}p^{\frac{1}{4}}(1-p)^{\frac{1}{12}}(1+p)^{\frac{1}{4}}(1+2p)^{\frac{1}{12}}(2+p)^{\frac{1}{12}}k^{\frac{1}{2}},\\
f_{12}&=2^{-\frac{2}{3}}q^{-\frac{1}{2}}p^{\frac{1}{2}}(1-p)^{\frac{1}{24}}(1+p)^{\frac{1}{8}}(1+2p)^{\frac{1}{24}}(2+p)^{\frac{1}{6}}k^{\frac{1}{2}}.
\end{align*}
\end{lemma}

\section{Congruences for $b(n)$}

\begin{theorem}
For $n\ge 0$, $\alpha\ge 1$, and prime $p\ge 5$, we have
\begin{align}
b\left(p^{2\alpha}n+\frac{(3j+p)p^{2\alpha-1}-1}{3}\right)\equiv0\pmod2,
\end{align}
where $j=1$, $2$, $\cdots$, $p-1$.
\end{theorem}

\begin{proof}
In light of \eqref{bn}, we derive that
\begin{align*}
\sum_{n=0}^\infty b(n)q^n=\frac{f_4^3}{f_2^2}\equiv f_8\pmod2.
\end{align*}

Applying Lemma \ref{cg13}, we deduce that, for any prime $p\geq5$,
$$\sum_{n=0}^\infty b\left(pn+\frac{p^2-1}{3}\right)q^n\equiv (-1)^{\frac{\pm p-1}{6}} f(-q^{8p})\pmod2,$$
and
$$\sum_{n=0}^\infty b\left(p^2n+\frac{p^2-1}{3}\right)q^n\equiv (-1)^{\frac{\pm p-1}{6}} f(-q^{8})\pmod2.$$
Moreover,
\begin{align*}
\sum_{n=0}^\infty b\left(p^{3}n+\frac{p^{4}-1}{3}\right)q^n&\equiv f(-q^{8p})\pmod2.
\end{align*}

Hence, by induction on $\alpha$, we derive that, for $\alpha\geq1$,
\begin{align*}
\sum_{n=0}^\infty b\left(p^{2\alpha-1}n+\frac{p^{2\alpha}-1}{3}\right)q^n&\equiv (-1)^{\alpha \left(\frac{\pm p-1}{6}\right)} f(-q^{8p})\pmod2.
\end{align*}
This immediately leads to
\begin{align*}
b\left(p^{2\alpha-1}(pn+j)+\frac{p^{2\alpha}-1}{3}\right)\equiv 0\pmod2,
\end{align*}
for $j=1$, $2$, $\cdots$, $p-1$. We complete the proof.
\end{proof}

\begin{remark}
When studying $1$-shell totally symmetric plane partition function $f(n)$ (which is different to Ramanujan's theta function $f(-q)$ given in Sect.~\ref{sec:2}) introduced by Blecher \cite{Ble2012}, Hirschhorn and Sellers \cite{Hirschhorn-Sellers-2014} proved that, for $n\ge 1$,
$$f(3n-2)=h(n),$$
with
$$\sum_{n=0}^\infty h(2n+1)q^n =\frac{f_2^3}{f_1^2}.$$

A couple of congruences modulo powers of $2$ and $5$ for $h(n)$ have been obtained subsequently; see \cite{Che2017,Xia15,Yao-2014}. We see from \eqref{bn} that
$$b(2n)=h(2n+1).$$
One therefore may obtain some congruences for $b(n)$ as well. For example,
\begin{align*}
b(8n+6)\equiv0\pmod4.
\end{align*}

%
\end{remark}

\section{Congruences for $c(n)$}

\begin{theorem}
For $n\geq 0$, we have
\begin{align}
c(27n+24)\equiv0\pmod9.
\end{align}
\end{theorem}

\begin{proof}
We see from \eqref{cn} and Lemma \ref{le:3-dis} that
\begin{align*}
\sum_{n=0}^\infty c(n)q^n&=\frac{qf_6^3}{\varphi(-q)\psi(q)}\nonumber\\
&=qf_6^3\frac{\varphi^3(-q^9)\psi^3(q^9)}{\varphi^4(-q^3)\psi^4(q^3)}
(1+2qw(q^3)+4q^2w^2(q^3))\nonumber\\
&\quad\quad\times\left(\frac{1}{w^2(q^3)}
-\frac{q}{w(q^3)}+q^2\right).
\end{align*}

Employing Lemma \ref{le:3-dis}, we deduce that
\begin{align}\label{eq:c3n-gf}
\sum_{n=0}^\infty c(3n)q^n&=\frac{3q\varphi^3(-q^3)\psi^3(q^3)}{f_1^3\varphi(-q)\psi(q)}\nonumber\\
&=\frac{3q\varphi^3(-q^9)\psi^3(q^9)f_9^3}{\varphi(-q^3)\psi(q^3)f_3^{12}}
\left(P^2(q^3)+3qP(q^3)f_9^3+9q^2f_9^6\right)\nonumber\\
&\quad\quad\times\left(1+2qw(q^3)+4q^2 w^2(q^3)\right)\left(\frac{1}{w^2(q^3)}-\frac{q}{w(q^3)}+q^2\right).
\end{align}

Extracting terms involving $q^{3n+2}$ and replacing $q^3$ by $q$ in \eqref{eq:c3n-gf}, it follows that
\begin{align*}
\sum_{n=0}^\infty c(9n+6)q^n=12\frac{f_2^2f_3^{21}}{f_1^{16}f_6^6}+135q\frac{f_3^{12}f_6^3}{f_1^{13}f_2}
+72q^2\frac{f_3^3f_6^{12}}{f_1^{10}f_2^4}+192q^3\frac{f_6^{21}}{f_1^7f_2^7f_3^6}.
\end{align*}
Hence,
\begin{align*}
\sum_{n=0}^\infty c(9n+6)q^n&\equiv 3\frac{f_2^2f_3^{21}}{f_1^{16}f_6^6}+3q^3\frac{f_6^{21}}{f_1^7f_2^7f_3^6}\\
&\equiv 3 \frac{f_2^2}{f_1}\left(\frac{f_3^{16}}{f_6^6}+
q^3\frac{f_6^{18}}{f_3^8}\right) \pmod{9}.
\end{align*}

Noting that $f_2^2/f_1$ contains no terms of the form $q^{3n+2}$, we have
\begin{align*}
\sum_{n=0}^\infty c(27n+24)q^n\equiv0\pmod9.
\end{align*}
It therefore ends the proof.
\end{proof}

\begin{theorem}\label{th:c45n}
For $n\geq 0$, we have
\begin{align}\label{c45n5}
c(45n+t)\equiv0\pmod5,
\end{align}
where $t=9$ and $18$.
\end{theorem}

\begin{proof}
Referring to \eqref{eq:c3n-gf}, we have
\begin{align*}
\sum_{n=0}^\infty c(9n)q^n=45q\frac{f_2f_3^{18}}{f_1^{15}f_6^3}+90q^2
\frac{f_3^9f_6^6}{f_1^{12}f_2^2}+288q^3\frac{f_6^{15}}{f_1^9f_2^5}.
\end{align*}
Hence,
\begin{align*}
\sum_{n=0}^\infty c(9n)q^n\equiv3q^3f_1\frac{f_{30}^{3}}{f_5^2f_{10}}\pmod5.
\end{align*}

Since $f_1$ contains no terms of the form $q^{5n+3}$ and $q^{5n+4}$, we have
$$c(9(5n+1))=c(45n+9)\equiv0\pmod5,$$
and
$$c(9(5n+2))=c(45n+18)\equiv0\pmod5.$$
This yields that \eqref{c45n5}.
\end{proof}


\begin{corollary}\label{co:c45n}
For $n\geq 0$, we have
\begin{align}
c(45n+t)\equiv0\pmod{15},
\end{align}
where $t=9$ and $18$.
\end{corollary}

\begin{proof}
We know from \cite[Theorem 4.2]{Choi-Kim-2012} that
$$c(3n)\equiv 0 \pmod{3}.$$
In fact, it is a direct consequence of \eqref{eq:c3n-gf}. Hence, Corollary \ref{co:c45n} follows by Theorem \ref{th:c45n}.
\end{proof}

\section{Congruences for $d(n)$}\label{sec:dn}

\begin{theorem}
For $n\geq0$, $\alpha\geq1$, and prime $p\ge 3$,
\begin{align}
d\left(2p^{2\alpha}+\frac{(8j+p)p^{2\alpha-1}-1}{4}\right)\equiv0\pmod2,
\end{align}
where $j=1$, $2$, $\cdots$, $p-1$.
\end{theorem}

\begin{proof}
From \eqref{dn}, one can see
\begin{align*}
\sum_{n=0}^\infty d(n)q^n=\frac{f_3^3}{f_1 f_2}\equiv f_6\frac{f_3}{f_1^3}\pmod2.
\end{align*}

With the help of \eqref{eq:f3f13}, we have
\begin{align*}
\sum_{n=0}^\infty d(n)q^n\equiv\frac{f_4^6f_6^4}{f_2^9f_{12}^2}+3q\frac{f_4^2f_6^2f_{12}^2}{f_2^7}\pmod2.
\end{align*}
Hence,
\begin{align*}
\sum_{n=0}^\infty d(2n)q^n\equiv\frac{f_2^6f_3^4}{f_1^9f_{6}^2}\equiv\psi(q)\pmod2.
\end{align*}

Invoking Lemma \ref{psi-p}, for any odd prime $p$, we derive that
\begin{align*}
\sum_{n=0}^\infty d \left(2\left(pn+\frac{p^2-1}{8}\right)\right)q^n\equiv\psi(q^p)\pmod2,
\end{align*}
and
\begin{align*}
\sum_{n=0}^\infty d \left(2\left(p^2n+\frac{p^2-1}{8}\right)\right)q^n\equiv\psi(q)\pmod2.
\end{align*}
Furthermore,
\begin{align*}
\sum_{n=0}^\infty d \left(2p^{3}n+\frac{p^4-1}{4}\right)q^n\equiv\psi(q^p)\pmod2.
\end{align*}

It therefore follows by induction on $\alpha$ that for $\alpha\geq1$,
\begin{align*}
\sum_{n=0}^\infty d \left(2p^{2\alpha-1}n+\frac{p^{2\alpha}-1}{4}\right)q^n\equiv\psi(q^p)\pmod2.
\end{align*}
Thus, for $j=1$, $2$, $\cdots$, $p-1$,
\begin{align*}
d \left(2p^{2\alpha-1}(pn+j)+\frac{p^{2\alpha}-1}{4}\right)\equiv0\pmod2,
\end{align*}
which is the desired result.
\end{proof}

\begin{theorem}
For $n\geq 0$, $\alpha\geq 1$, and prime $p\geq 5$, we have
\begin{align}\label{cmod3}
d\left(6p^{2\alpha}n+\frac{(24j+p)p^{2\alpha-1}-1}{4}\right)\equiv0\pmod3,
\end{align}
where $j=1$, $2$, $\cdots$, $p-1$.
\end{theorem}

\begin{proof}
It follows by \eqref{eq:1/-phi--3dis} and \eqref{eq:1/psi--3dis} that
\begin{align}\label{d-n}
\sum_{n=0}^\infty d(n)&=\frac{f_3^3}{\varphi(-q)\psi(q)}\nonumber\\
&=f_3^3 \frac{\varphi^3(-q^9)\psi^3(q^9)}{\varphi^4(-q^3)\psi^4(q^3)}\left(1+2qw(q^3)+4q^2 w^2(q^3)\right)\nonumber\\
&\quad\quad\times\left(\frac{1}{w^2(q^3)}
-\frac{q}{w(q^3)}+q^2\right).
\end{align}

So we get
\begin{align*}
\sum_{n=0}^\infty d(3n)q^n&=f_1^3\frac{\varphi^3(-q^3)\psi^3(q^3)}{\varphi^4(-q)\psi^4(q)}
\left(\frac{1}{w^2(q)}-2qw(q)\right)\\
&=\frac{1}{f_2^2f_6^3}\left(\frac{f_3^3}{f_1}\right)^3-2q\frac{f_6^6}{f_2^5}.
\end{align*}
Based on \eqref{eq:f33f1}, we derive that
\begin{align*}
\sum_{n=0}^\infty d(6n)q^n=\frac{f_2^9f_3^3}{f_1^8f_6^3}+3q\frac{f_2f_6^5}{f_1^2f_3}\equiv\frac{f_2^9f_3^3}{f_1^8f_6^3}
\equiv f_1\pmod3.
\end{align*}

Invoking Lemma \ref{cg13}, we arrive at that, for any prime $p\geq5$,
\begin{align*}
\sum_{n=0}^\infty d\left(6\left(pn+\frac{p^2-1}{24}\right)\right)q^n\equiv(-1)^{\frac{\pm p-1}{6}} f(-q^p)\pmod3,
\end{align*}
and
\begin{align*}
\sum_{n=0}^\infty d\left(6\left(p^2n+\frac{p^2-1}{24}\right)\right)q^n\equiv (-1)^{\frac{\pm p-1}{6}} f(-q)\pmod3.
\end{align*}
Furthermore, we have
\begin{align*}
\sum_{n=0}^\infty d\left(6\left(p^2\left(pn+\frac{p^2-1}{24}\right)+\frac{p^2-1}{24}\right)\right)q^n\equiv f(-q^p)\pmod3.
\end{align*}
Namely,
\begin{align*}
\sum_{n=0}^\infty d\left(6p^3n+\frac{p^4-1}{4}\right)q^n\equiv f(-q^p)\pmod3.
\end{align*}

Thus, by induction on $\alpha$, we derive that, for $\alpha\geq1$,
\begin{align*}
\sum_{n=0}^\infty d\left(6p^{2\alpha-1}n+\frac{p^{2\alpha}-1}{4}\right)q^n\equiv (-1)^{\alpha\left(\frac{\pm p-1}{6}\right)} f(-q^p)\pmod3.
\end{align*}
This yields that, for $j=1$, $2$, $\cdots$, $p-1$,
\begin{align*}
d\left(6p^{2\alpha-1}(pn+j)+\frac{p^{2\alpha}-1}{4}\right)\equiv0\pmod3,
\end{align*}
which implies \eqref{cmod3}.
\end{proof}

\begin{theorem}\label{th:d-mod9}
For $n\geq0$, $\alpha\geq1$, and prime $p\ge 5$,
\begin{align}\label{d-mod9}
d\left(6p^{2\alpha}n+\frac{(24j+9p)p^{2\alpha-1}-1}{4}\right)\equiv 0\pmod9,
\end{align}
where $j=1$, $2$, $\cdots$, $p-1$.
\end{theorem}

\begin{proof}
Extracting terms involving $q^{3n+2}$ and replace $q^3$ by $q$ in \eqref{d-n}, then we derive that
\begin{equation}\label{eq:d3n+2-gf}
\sum_{n=0}^\infty d(3n+2)q^n=\frac{3f_3^3 f_6^3}{\varphi(-q)\psi(q)f_2^3}=\frac{3f_3^3 f_6^3}{f_1 f_2^4}.
\end{equation}

It follows by \eqref{eq:f33f1} that,
\begin{align*}
\sum_{n=0}^\infty d(3n+2)q^n=3\frac{f_3^3f_6^3}{f_1f_2^4}=3\frac{f_4^3f_6^5}{f_2^6f_{12}}
+3q\frac{f_6^3f_{12}^3}{f_2^4f_4}.
\end{align*}
Hence,
\begin{align*}
\sum_{n=0}^\infty d(6n+2)q^n=3\frac{f_2^3f_3^5}{f_1^6f_6}\equiv 3f_9\pmod{9}.
\end{align*}

In view of Lemma \ref{cg13}, for any prime $p\geq5$, we deduce that
$$\sum_{n=0}^\infty d\left(6\left(pn+\frac{3(p^2-1)}{8}\right)+2\right)q^n\equiv 3(-1)^{\frac{\pm p-1}{6}} f(-q^{9p})\pmod9,$$
and
$$\sum_{n=0}^\infty d\left(6\left(p^2n+\frac{3(p^2-1)}{8}\right)+2\right)q^n\equiv 3(-1)^{\frac{\pm p-1}{6}} f(-q^9)\pmod9.$$
Moreover,
\begin{align*}
\sum_{n=0}^\infty d\left(6\left(p^3n+\frac{3(p^4-1)}{8}\right)+2\right)q^n&\equiv 3f(-q^{9p})\pmod9.
\end{align*}

Hence, by induction on $\alpha\ge 1$, we arrive at,
\begin{align*}
\sum_{n=0}^\infty d\left(6\left(p^{2\alpha-1}n+\frac{3(p^{2\alpha}-1)}{8}\right)+2\right)q^n\equiv 3(-1)^{\alpha\left(\frac{\pm p-1}{6}\right)} f(-q^{9p})\pmod9,
\end{align*}
which implies that for $j=1$, $2$, $\cdots$, $p-1$,
\begin{align*}
d\left(6\left(p^{2\alpha-1}(pn+j)+\frac{3(p^{2\alpha}-1)}{8}\right)+2\right)\equiv 0\pmod9.
\end{align*}
This leads to \eqref{d-mod9}.
\end{proof}

\begin{theorem}\label{th:d-mod5}
For $n\geq0$, we have
\begin{equation}
d(45n+t)\equiv 0 \pmod{5},
\end{equation}
where $t=17$ and $35$.
\end{theorem}

\begin{proof}
From \eqref{eq:d3n+2-gf}, we have
$$\sum_{n=0}^\infty d(3n+2)q^n=\frac{3f_3^3 f_6^3}{\varphi(-q)\psi(q)f_2^3}.$$

Again by \eqref{eq:1/-phi--3dis}, \eqref{eq:1/psi--3dis} and \eqref{eq:1/f13--3dis}, we have
$$\sum_{n=0}^\infty d(9n+8)q^n=f_2\cdot H,$$
where
\begin{align*}
H&=\left(\frac{9f_3^9 f_4 f_6^9}{f_1^3 f_2^{13} f_{12}^3}+\frac{9 f_3^3 f_4^2 f_6^{18}}{f_1 f_2^{16} f_{12}^6}\right)+q\left(\frac{27 f_3^6 f_6^9}{f_1^2 f_2^{13}}-\frac{18 f_4 f_6^{18}}{f_2^{16} f_{12}^3}\right)\\
&\quad\quad+q^2\left(\frac{36 f_3^9 f_{12}^6}{f_1^{3} f_2^{10} f_4^2}+\frac{72 f_3^3 f_6^9 f_{12}^3}{f_1 f_2^{13} f_4}+\frac{108 f_1 f_6^{18}}{f_{2}^{16} f_3^3}\right)\\
&\quad\quad-q^3 \frac{72 f_6^9 f_{12}^6}{f_{2}^{13} f_4^2}+q^4\frac{144 f_3^3 f_{12}^{12}}{f_1 f_2^{10} f_4^4}.
\end{align*}

We next show a surprising congruence.
\begin{lemma}\label{le:H}
It holds that
\begin{equation}\label{eq:H-mod5}
H\equiv 3 \frac{f_{15}^3}{f_5 f_{10}^2} \pmod{5}.
\end{equation}
\end{lemma}

\begin{proof}[Proof of Lemma \ref{le:H}]
To prove \eqref{eq:H-mod5}, it suffices to show
$$H-3 \frac{f_3^{15}}{f_1^5 f_2^{10}}\equiv 0 \pmod{5},$$
or equivalently,
\begin{align*}
\left(H-3 \frac{f_3^{15}}{f_1^5 f_2^{10}}\right)\frac{f_1^5 f_3 f_4^{10} f_6^{10}}{f_2^4 f_{12}^6}\equiv 0 \pmod{5},
\end{align*}
since $\frac{f_1^5 f_3 f_4^{10} f_6^{10}}{f_2^4 f_{12}^6}$ is invertible in the ring $\mathbb{Z}/5\mathbb{Z}[[q]]$. According to Lemma \ref{le:AW}, it becomes
$$\frac{15 p^2 (1-p) (1+p)^5 (2+p)^2 (2+5p +12 p^2+ 5 p^3+ 2 p^4) k^8}{32 q^2 (1+2p)}\equiv 0 \pmod{5}.$$
Lemma \ref{le:H} follows obviously.
\end{proof}

We know from Lemma \ref{le:H} that
$$\sum_{n=0}^\infty d(9n+8)q^n\equiv3 f_2 \frac{f_{15}^3}{f_5 f_{10}^2} \pmod{5}.$$
Since $f_2=(q^2;q^2)_\infty$ contains no terms of the form $q^{5n+1}$ and $q^{5n+3}$, we have
$$d(9(5n+1)+8)=d(45n+17)\equiv 0 \pmod{5},$$
and
$$d(9(5n+3)+8)=d(45n+35)\equiv 0 \pmod{5},$$
which leads to Theorem \ref{th:d-mod5}.
\end{proof}

\begin{corollary}\label{co:d45n}
For $n\geq0$, we have
\begin{equation}
d(45n+t)\equiv 0 \pmod{15},
\end{equation}
where $t=17$ and $35$.
\end{corollary}

\begin{proof}
Again, we know from \cite[Theorem 4.2]{Choi-Kim-2012} that
$$d(3n+2)\equiv 0 \pmod{3}.$$
It indeed follows directly from \eqref{eq:d3n+2-gf}. We therefore prove Corollary \ref{co:d45n} by Theorem \ref{th:d-mod5}.
\end{proof}

\section{Final remarks}

We end this paper by raising the following congruences.

\begin{question}
We have
\begin{align}
c(45n+21)&\equiv0\pmod5,\\
c(63n+t)&\equiv0\pmod7,
\end{align}
where $t=30$, $48$ and $57$.
\end{question}

\begin{question}
We have
\begin{align}
d(45n+41)&\equiv0\pmod5,\\
d(63n+t)&\equiv0\pmod7,
\end{align}
where $t=32$, $50$ and $59$.
\end{question}

All these congruences have been verified by the authors using an algorithm due to Radu and Sellers \cite{RS2011}. However, since the modular form proofs are very routine and tedious, we here want to ask if there exist elementary proofs of these congruences.


\bibliographystyle{amsplain}

\begin{thebibliography}{99}

\bibitem{AW2010}
\c{S}. Alaca and K. S. Williams, The number of representations of a positive integer by certain octonary quadratic forms, \textit{Funct. Approx. Comment. Math.} \textbf{43} (2010), part 1, 45--54.


\bibitem{Andrews-2005}
G. E. Andrews, Partitions with short sequences and mock theta functions, \textit{Proc. Natl. Acad. Sci. USA} \textbf{102} (2005), no. 13, 4666--4671.

\bibitem{Andrews-Dixit-Schultz-Yee-2015}
G. E. Andrews, A. Dixit, D. Schultz, and A. J. Yee, Overpartitions related to the mock theta function $\omega(q)$, \textit{Preprint} (2016). Available at arXiv:1603.04352.

\bibitem{Andrews-Dixit-Yee-2015}
G. E. Andrews, A. Dixit, and A. J. Yee, Partitions associated with the Ramanujan/Watson mock theta functions $\omega(q)$, $\nu(q)$ and $\phi(q)$, \textit{Res. Number Theory} \textbf{1} (2015), Art. 19, 25 pp.

\bibitem{Andrews-Garvan-1989}
G. E. Andrews and F. G. Garvan, Ramanujan's ``lost'' notebook. VI. The mock theta conjectures, \textit{Adv. in Math.} \textbf{73} (1989), no. 2, 242--255.

\bibitem{Andrews-Passary-Sellers-Yee-2017}
G. E. Andrews, D. Passary, J. Seller, and A. J. Yee, Congruences related to the Ramanujan/Watson mock theta functions $\omega(q)$ and $\nu(q)$, \textit{Ramanujan J.} \textbf{43} (2017), no. 2, 347--357.

\bibitem{Baruah-Ojah-2011}
N. D. Baruah and K. K. Ojah, Some congruences deducible from Ramanujan's cubic continued fraction, \textit{Int. J. Number Theory} \textbf{7} (2011), no. 5, 1331--1343.

\bibitem{Ber1991}
B. C. Berndt, \textit{Ramanujan's notebooks. Part III}, Springer-Verlag, New York, 1991. xiv+510 pp.

\bibitem{BR1995}
B. C. Berndt and R. A. Rankin, \textit{Ramanujan: Letters and commentary}, American Mathematical Society, London Mathematical Society, History of Mathematics Series, Vol. 9, 1995, 347 pp.

\bibitem{Ble2012}
A. Blecher, Geometry for totally symmetric plane partitions (TSPPs) with self-conjugate main diagonal, \textit{Util. Math.} \textbf{88} (2012), 223--235.

\bibitem{Chan-2010-1}
H.-C. Chan, Ramanujan's cubic continued fraction and an analog of his ``most beautiful identity'', \textit{Int. J. Number Theory} \textbf{6} (2010), no. 3, 673--680.

\bibitem{Chan-2010-2}
H.-C. Chan, Ramanujan's cubic continued fraction and Ramanujan type congruences for a certain partition function, \textit{Int. J. Number Theory} \textbf{6} (2010), no. 4, 819--834.

\bibitem{Che2017}
S. Chern, Congruences for $1$-shell totally symmetric plane partitions, \textit{Integers} \textbf{17} (2017), Paper No. A21, 7 pp.

\bibitem{Choi-2011}
Y.-S. Choi, The basic bilateral hypergeometric series and the mock theta functions, \textit{Ramanujan J.} \textbf{24} (2011), no. 3, 345--386.

\bibitem{Choi-Kim-2012}
Y.-S. Choi and B. Kim, Partition identities from third and sixth order mock theta functions, \textit{European J. Combin.} \textbf{33} (2012), no. 8, 1739--1754.


\bibitem{Cui-Gu-2013}
S.-P. Cui and N. S. S. Gu, Arithmetic properties of $\ell$-regular partitions, \textit{Adv. in Appl. Math.} \textbf{51} (2013), no. 4, 507--523.

\bibitem{Garvan-Kim-Stanton-1990} F. Garvan, D. Kim, and D. Stanton, Cranks and $t$-cores, \textit{Invent. Math.} \textbf{101} (1990), no. 1, 1--17.

\bibitem{Gasper-Rahman-2004}
G. Gasper and M. Rahman, \textit{Basic hypergeometric series. Second edition}, Encyclopedia of Mathematics and its Applications, 96. Cambridge University Press, Cambridge, 2004. xxvi+428 pp.

\bibitem{Hirschhorn-Sellers-2014}
M. D. Hirschhorn and J. A. Sellers, Arithmetic properties of $1$-shell totally symmetric plane partitions, \textit{Bull. Aust. Math. Soc.} \textbf{89} (2014), no. 3, 473--478.

\bibitem{Kim-2011}
B. Kim, An analog of crank for a certain kind of partition function arising from the cubic continued fraction, \textit{Acta Arith.} \textbf{148} (2011), no. 1, 1--19.

\bibitem{RS2011}
S. Radu and J. A. Sellers, Congruence properties modulo $5$ and $7$ for the $\mathrm{pod}$ function, \textit{Int. J. Number Theory} \textbf{7} (2011), no. 8, 2249--2259.

\bibitem{She1994}
L.-C. Shen, On the modular equations of degree 3, \textit{Proc. Amer. Math. Soc.} \textbf{122} (1994), no. 4, 1101--1114.

\bibitem{Wang-2016}
L. Wang, Arithmetic identities and congruences for partition triples with 3-cores, \textit{Int. J. Number Theory} \textbf{12} (2016), no. 4, 995--1010.

\bibitem{Xia15}
E. X. W. Xia, A new congruence modulo $25$ for $1$-shell totally symmetric plane partitions, \textit{Bull. Aust. Math. Soc.} \textbf{91} (2015), no. 1, 41--46.

\bibitem{Xia-Yao-2013}
E. X. W. Xia and O. X. M. Yao, Analogues of Ramanujan's partition identities, \textit{Ramanujan J.} \textbf{31} (2013), no. 3, 373--396.

\bibitem{Yao-2014}
O. X. M. Yao, New infinite families of congruences modulo $4$ and $8$ for $1$-shell totally symmetric plane partitions, \textit{Bull. Aust. Math. Soc.} \textbf{90} (2014), no. 1, 37--46.


\end{thebibliography}

\end{document}